\documentclass[11pt,reqno]{amsart}

\usepackage{amssymb}
\usepackage[all]{xy}
\usepackage[inline]{enumitem}

 \newtheorem{theorem}{Theorem}[section]
 \newtheorem{proposition}[theorem]{Proposition}
 \newtheorem{corollary}[theorem]{Corollary}

 \newtheorem{lemma}[theorem]{Lemma}
 \newtheorem{example}[theorem]{Example}

 \newcommand{\ev}{\mathop{\rm ev}\nolimits}
 \newcommand{\cat}{\mathop{\rm cat}\nolimits}
 \newcommand{\nil}{\mathop{\rm nil}\nolimits}
 \newcommand{\TC}{\mathop{\rm TC}\nolimits}
 
 \newcommand{\height}{\mathop{\rm height}\nolimits}
 
 \newcommand{\Ker}{\mathop{\rm Ker}\nolimits}
 
 \newcommand{\zcl}{\mathop{\rm zcl}\nolimits}

 \newcommand{\ZZ}{{\mathbb{Z}}}
 
 \newcommand{\RR}{{\mathbb{R}}}
 \newcommand{\CC}{{\mathcal{C}}}

 \newcommand{\gkn}{G_k(\RR^n)}
 \newcommand{\fln}{\mathrm{Flag}(\RR^n)}
 \newcommand{\ow}{\overline{w}}

\begin{document}

\title{Topological complexity of real Grassmannians} 
\author[Petar Pave\v{s}i\'{c}]{Petar Pave\v si\'{c}$^{*}$}
\address{$^{*}$ Faculty of Mathematics and Physics, University of Ljubljana,
Jadranska 21,  1000 Ljubljana, Slovenija}
\email{petar.pavesic@fmf.uni-lj.si}
\thanks{$^{*}$ Supported by the Slovenian Research Agency program P1-0292 and grants N1-0083, N1-0064}

\date{\today}

\begin{abstract} 
We use some detailed knowledge of the cohomology ring of real Grassmann manifolds $\gkn$ 
to compute zero-divisor cup-length and estimate topological complexity of motion planning for $k$-linear subspaces in $\RR^n$.  
In addition, we obtain results about monotonicity of Lusternik-Schnirelmann category 
and topological complexity of $\gkn$ as a function of $n$. 
\ \\[3mm]
\end{abstract}

\maketitle

\section{Introduction}

Topological complexity is a homotopy invariant of a space that was introduced by M. Farber 
\cite{Farber:TC} in his study of motion planning in robotics. Roughly speaking, motion planning 
problem requires to 
find a continuous motion that transforms a mechanical system from some
given initial position to a desired final position. To give a precise mathematical 
statement one considers paths in a suitable \emph{configuration space}, i.e. the topological space
that describes all possible states of the mechanical system. In particular, Grassmann
manifolds of linear subspaces of $\RR^n$ arise naturally in the motion planning for geometric objects like lines or planes in Euclidean spaces. 

Given a path-connected configuration space $X$ of a mechanical system, we denote by  $X^I$ the space of all 
continuous paths $\alpha\colon I\to X$, and by $\ev \colon X^I\to X\times X$
the evaluation map  that to a path 
$\alpha$ assigns its end-points, $\ev(\alpha):=(\alpha(0),\alpha(1))$. 
A \emph{motion plan} in $X$ is a function that takes as input a pair of points 
$x,y\in X$, and returns as output a path $\alpha=\alpha(x,y)$ in $X$ starting at $x$ and ending 
at $y$. That is to say, a motion plan is a section $\alpha\colon X\times X\to X^I$ 
of the evaluation map. It is important to note that while the motion through the
configuration space is always assumed to be continuous, motion plans will be 
generally discontinuous.  In fact, one can easily show that a continuous motion plan
exists if and only if $X$ is contractible. In applications discontinuities appear 
as instabilities of the motion plan because small variations or errors in the input 
data can result in a completely different motion path of a robot. 
One is thus naturally led to consider motion plans whose degree of instability is as 
small as possible. 

Farber \cite{Farber:TC} defined the topological complexity $\TC(X)$ of a path-connected space $X$ 
as the minimal integer $n$ for which there exists an open cover $U_1,\ldots,U_n$ of 
$X\times X$ and a motion plan $\alpha\colon X\times X\to X^I$, such that the restrictions 
of $\alpha$ to each $U_i$ is continuous. We will refer to Farber's monograph 
\cite[Chapter 4]{Farber:Invitation} for the relevant definitions and results 
on $\TC(X)$. In particular, see \cite[Section 4.2]{Farber:Invitation}
for several alternative descriptions of topological complexity and for its relation 
with the instability of motion plans. 

The Grasmann manifold of 1-dimensional subspaces of $\RR^n$ is the familiar projective
space $\RR P^{n-1}$. Topological complexity of projective spaces was studied by 
Farber, Tabachnikov and Yuzvinsky \cite{FTY:TC of RPn}, who found a remarkable
relation between $\TC(\RR P^n)$ and the classical immersion dimension of $\RR P^n$, which gave 
precise values for the topological complexity of real projective spaces in a range of dimensions. Soon afterwards 
Pearson and Zhang \cite{Pearson-Zhang} published a short note in which they produced some estimates 
for the topological 
complexity of higher Grassmannians. Unfortunately, their work was marred by a trivial
error that led them to assume that $\TC(X)=\cat(X\times X)$ (while the correct relation is 
$\cat(X)\le \TC(X)\le\cat(X\times X)$, see below) which compromised all of their
results.  Note that even after correcting the error 
in \cite{Pearson-Zhang}, the methods used in that paper would give an estimate for $\TC(\gkn)$ that 
is only slightly better than the standard category estimate. 
Through a careful analysis of the cohomology of Grassmannians we were able to find a different set of non-trivial
products and derive new lower bounds $\TC(\gkn)$ that are from two to three times better 
than those that were previously known.

The next two sections contain preparatory material. In the first we briefly review the main
concepts of Lusternik-Schnirelmann category and topological complexity, while in the second 
we summarize some technical results on Grassmann manifolds and their cohomology rings, and 
prove the non-triviality of certain cohomology classes which will serve as a base for our 
computations.
Section 4 is the central part of the article and is dedicated to the estimates of topological 
complexity of Grassmannians. The main results
are Theorems \ref{thm:g2n}, \ref{thm:g3n}, \ref{thm:gkn}  and \ref{thm:g4n} where explicit lower bounds for 
the topological 
complexity of manifolds $G_k(\RR^n)$ are given for $k=2$, $k=3$ and $k>3$. 
In the last section we use the above estimates to show that under suitable
assumptions the Lusternik-Schnirelmann category and the topological complexity of $\gkn$
are increasing functions of $n$.  

\ \\

\section{Review of category and topological complexity}
\label{sec:LScat & TC}

Here we recall main properties of Lusternik-Schnirelmann category and 
topological complexity. Note that in both cases we follow the non-normalized convention,
i.e., category and complexity of contractible spaces is equal to 1 (as opposed to 0 
in the normalized case). 
More details can be found in the monographs \cite{CLOT} and \cite{Farber:Invitation}.

The \emph{Lusternik-Schnirelmann category} of a space $X$, denoted $\cat(X)$, is the minimal 
integer $n$ for which there exists an open cover $U_1,\ldots,U_n$ of $X$ such that each $U_i$ can be
deformed in $X$ to a point. Main properties of $\cat(X)$ are listed in the following
proposition.

\begin{proposition}\ \\[-6mm]
\label{prop:LScat properties}
\begin{enumerate}
\item $\cat(X)=1$ if, and only if $X$ is contractible;
\item Homotopy invariance: $X\simeq Y \Rightarrow \cat(X)=\cat(Y)$;
\item Dimension-connectivity estimate: if $X$ is $d$-dimensional and $(c-1)$-connected, then 
$\cat(X)\le\frac{d}{c}+1$;
\item Cohomological estimate: $\cat(X)\ge \nil\widetilde H^*(X)$, where 
$\widetilde H^*(X)$ is the ideal of positive-dimensional cohomology classes in $H^*(X)$;
\item Product formula: $\cat(X\times Y)\le \cat(X)+\cat(Y)-1$.
\end{enumerate}
\end{proposition}

In (4) $H^*(X)$ denotes cohomology with any ring coefficients, and the nilpotency 
$\nil\widetilde H^*(X)$ of the ideal $\widetilde H^*(X)$ is
the minimal integer $n$ for which $\widetilde H^*(X)^n=0$ (i.e., all products of $n$ elements of 
$\widetilde H^*(X)$ are equal to 0). 
The invariant $\cat(X)$ was extensively studied and it was computed exactly or to a good approximation
for a wide variety of spaces (see \cite{CLOT}). However, the knowledge of the LS-category of real Grassmannians, which is of interest for this paper, is still very scant. 

We have already mentioned Farber's topological complexity and its applications to
robotics. The following definition displays a strong analogy with Lusternik-Schnirelmann category 
(and is easily seen to be equivalent to the one we gave earlier in terms of continuous motion plans).
The \emph{topological complexity} of a path-connected space $X$, denoted $\TC(X)$, is the minimal 
integer $n$ for which there exists an open cover $U_1,\ldots,U_n$ of $X\times X$ such that each $U_i$ can be
deformed in $X\times X$ to the diagonal $\Delta X$. Again we collect the main properties of this concept 
in a proposition.

\begin{proposition}\ \\[-6mm]
\label{prop:TC properties}
\begin{enumerate}
\item $\TC(X)=1$ if, and only if $X$ is contractible;
\item Homotopy invariance: $X\simeq Y \Rightarrow \TC(X)=\TC(Y)$;
\item Category estimate: $\cat(X)\le \TC(X)\le \cat(X\times X)$;
\item If $X$ is a topological group, then $\TC(X)=\cat(X)$;
\item Cohomological estimate: $\TC(X)\ge \nil(\Ker\Delta^*)$, where 
$$\Delta^*\colon H^*(X\times X)\to H^*(X)$$ 
is induced by the diagonal $\Delta\colon X\to X\times X$;
\item Product formula: $\TC(X\times Y)\le \TC(X)+\TC(Y)-1$.
\end{enumerate}
\end{proposition}

We will also need two special results. 

\begin{proposition}(see \cite[Theorem 8.23]{CLOT} and \cite[Lemma 28.1]{Farber:TRMP}) 
\label{prop:symplectic}
Let $X$ be a simply-connected symplectic manifold. Then $\cat(X)=\dim(X)/2+1$ and $\TC(X)=\dim(X)+1$. 
\end{proposition}

\begin{proposition}(see \cite[Theorem 1]{Costa-Farber}) 
\label{prop:small pi1}
If $X$ is a CW-complex with $\pi_1(X)=\ZZ_2$, then $\TC(X)\le 2\dim X$. 
Furthermore, if $X$ is a closed $n$-manifold, and the generator $w\in H^1(X;\ZZ_2)=\ZZ_2$ 
satisfies $w^n=0$, then $\TC(X)\le 2\dim X-1$. 
\end{proposition}

\ \\

\section{Review of finite Grassmannians and their cohomology}
\label{sec:Grassmannians}

Let $\gkn$ denote the space of all $k$-dimensional linear subspaces of $\RR^n$. 
This space is a manifold (actually a non-singular algebraic variety) of dimension $k(n-k)$. The
correspondence between a subspace and its orthogonal complement determines a homeomorphism between 
$\gkn$ and $G_{n-k}(\RR^n)$, so we will always tacitly assume that $k\le n/2$.

The space $\gkn$ has a standard CW-decomposition which we briefly recall 
following \cite[Section 6]{Milnor-Stasheff}. For each element (i.e., $k$-subspace of $\RR^n$) 
$V\in\gkn$ consider the increasing sequence of integers
$$0\le\dim(V\cap \RR)\le\dim(V\cap \RR^2)\le\cdots\le\dim(V\cap \RR^n)=k,$$
where two consecutive terms are either equal or differ by 1, so that there are exactly
$k$ 'jumps' of the form $\dim(V\cap \RR^{d_{i}})=\dim(V\cap \RR^{d_{i-1}})+1$ for $i=1,2,\ldots,k$. 
The \emph{Schubert symbol} of $V$ is defined as a sequence of integers
$$\sigma(V)=(d_1-1,d_2-2,\ldots,d_k-k).$$
Note that for every Schubert symbol $\sigma(V)=(\sigma_1,\ldots,\sigma_k)$ we have 
$0\le\sigma_1\le\ldots\le\sigma_k\le n-k$, and that every such sequence can appear for some $V\in \gkn$.
Thus, for every non-increasing sequence $\sigma$ of integers between 0 and $n-k$ let
$$e_\sigma:=\{V\in\gkn\mid \sigma(V)=\sigma\}.$$
It turns out (see \cite[Theorem 6.4]{Milnor-Stasheff}) that $e_\sigma$ is homeomorphic to an open cell
of dimension $\sigma_1+\ldots+\sigma_k$, and that these cells form a CW-decomposition of $\gkn$. 
There is one 0-cell corresponding to the symbol $(0,\ldots,0)$ and one top-dimensional 
$k(n-k)$-cell with symbol $(n-k,\ldots,n-k)$. In general the number of $d$-dimensional
cells is equal to the number of partitions of $d$ as a sum of $k$ integers between 0 and $n-k$,
so that there are  ${n \choose k}$ cells in total.

There is a natural embedding $\gkn\hookrightarrow G_k(\RR^{n+1})$ obtained by 
viewing every $k$-dimensional subspace of $\RR^n$ as a subspace of $\RR^n\oplus\RR=\RR^{n+1}$.
It is easy to see that $\gkn$ is actually a subcomplex of $G_k(\RR^{n+1})$ with respect to the 
Schubert decomposition and that cells up to dimension $(n-k)$ coincide, i.e.
$$\gkn^{(n-k)}=G_k(\RR^{n+1})^{(n-k)}.$$

We will base our computations on the $\ZZ_2$-cohomology ring of $\gkn$ which can be described 
as follows (cf. \cite{Stong}). 
Let $w_1,\ldots,w_k\in H^*(\gkn;\ZZ_2)$ denote the Stiefel-Whitney classes of the canonical
$k$-dimensional vector bundle over $\gkn$, and let $\ow_1,\ldots,\ow_{n-k}\in H^*(\gkn;\ZZ_2)$
denote the dual  Stiefel-Whitney classes 
(i.e., the Stiefel-Whitney classes of the orthogonal complement of the canonical bundle).
The Stiefel-Whitney classes and their duals are related by the formula
$$w\cdot\ow=(1+w_1+\ldots + w_k)\cdot(1+\ow_1+\ldots+\ow_{n-k})=1, $$ 
so the dual classes can be recursively expressed as polynomials in the variables $w_1,\ldots,w_k$.
Then $H^*(\gkn;\ZZ_2)$ can be described as the quotient of the $\ZZ_2$-polynomial ring,
generated by the Stiefel-Whitney classes and their duals, modulo the relation $w\cdot \ow=1$:
$$ H^*(\gkn;\ZZ_2)\cong\ZZ_2[w_1,\ldots,w_k,\ow_1,\ldots,\ow_{n-k}]/(w\cdot \ow=1).$$

\begin{example}
In the cohomology of $G_2(\RR^6)$ we have the relation 
$$(1+w_1+w_2)\cdot (1+\ow_1+\ow_2+\ow_3+\ow_4)=1,$$ 
from which 
we obtain $\ow_1=w_1,\ \ow_2=w_1^2+w_2,\  \ow_3=w_1^3,\ \ow_4=w_1^4+w_1^2w_2+w_2^2,
\  w_1^5+w_1w_2^2=0$ and $w_1^4w_2+w_1^2w_2^2+w_2^3=0$. Therefore
$$ H^*(G_2(\RR^6);\ZZ_2)\cong\ZZ_2[w_1,w_2]/(w_1^5+w_1w_2^2,\ w_1^4w_2+w_1^2w_2^2+w_2^3).$$
\end{example}

In spite of a very explicit description, computations in $H^*(\gkn;\ZZ_2)$ are all but straightforward,
because it is often extremely difficult to determine whether a given polynomial in $w_1,\ldots,w_k$ is 
contained in the
ideal. R.~Stong \cite[p. 103-104]{Stong} (following earlier work by H.L.~Hiller) used a different
approach to check if a given product of Stiefel-Whitney classes is non-trivial in $H^*(\gkn;\ZZ_2)$.
He considered the natural map 
$$\pi\colon\fln\to\gkn,$$
where $\fln$ is the space of all ordered $n$-tuples $(V_1,\ldots,V_n)$ ('flags') of mutually 
orthogonal 1-dimensional subspaces of $\RR^n$ and $\pi(V_1,\ldots,V_n)=V_1\oplus\ldots\oplus V_k)$. 
The main facts about the cohomology of $\fln$ and the homomorphism $\pi^*$ are collected 
in the following proposition.

\begin{proposition}(see Stong \cite[pp.~104-106]{Stong})
\label{prop:prop pi*}
\begin{enumerate}
\item $H^*(\fln;\ZZ_2)\cong\ZZ_2[e_1,\ldots,e_n]/(\prod_{i=1}^n (1+e_i)=1)$, where $e_i=w_1(L_i)$ for 
certain canonical line bundles $L_1,\ldots, L_n$ over $\fln$.
\item $e_i^n=0$ in $H^*(\fln;\ZZ_2)$ for $i=1,\ldots,n$.
\item a monomial $e_1^{i_1}\ldots e_n^{i_n}$ in the top dimension $H^{n \choose 2}(\fln;\ZZ_2)$ is 
non-zero if, and only if, all exponents are different, i.e., if $(i_1,\ldots i_n)$ is a permutation
of $\{0,\ldots, n-1\}$.
\item $\pi^*\colon H^*(\gkn;\ZZ_2)\to H^*(\fln;\ZZ_2)$ is injective and is given explicitly as
$$\pi^*(1+w_1+\ldots+w_k)=\prod_{i=1}^k(1+e_i),\ \ \text{and}\ \ 
\pi^*(1+\ow_1+\ldots+\ow_{n-k})=\prod_{i=k+1}^n(1+e_i);$$
In other words, $\pi^*(w_i)$ is the $i$-th elementary symmetric polynomial in the variables 
$e_1,\ldots,e_k$ and $\pi^*(\ow_i)$ is the $i$-th elementary symmetric polynomial in the variables 
$e_{k+1},\ldots,e_n$.
\item An element $x\in H^*(\gkn;\ZZ_2)$ is non-zero if, and only if, the product
$$\pi^*(x)\cdot (e_1^{k-1}e_2^{k-2} \ldots e_{k-1})\cdot(e_{k+1}^{n-k-1}\ldots e_{n-1})\in 
H^*(\fln;\ZZ_2)$$
is non-zero.
\end{enumerate}
\end{proposition}

Based on these properties Stong \cite[p. 103-104]{Stong} proved a series of results about
non-triviality of products in $H^*(\gkn;\ZZ_2)$. He used them to give what are still the best known 
estimates of the LS-category of real Grassmannians. For the convenience of the reader we first summarize 
Stong's results and then derive further non-trivial
cohomology products that will lead to new and much stronger estimates of topological complexity.

\begin{proposition}(Stong, \cite[p.~103]{Stong})
\label{prop:height w1}
If $k\ge 2$ and $2^s< n\le 2^{s+1}$ then the \emph{height} of $w_1$ (i.e., the maximal $m$, such that
$w_1^m\ne 0$ in $ H^*(\gkn;\ZZ_2)$) is 
$$\mathrm{height}(w_1)=\left\{\begin{array}{ll}
2^{s+1}-2 & k=2 \mathrm{\ or \ } k=3, n=2^s+1\\
2^{s+1}-1 & \mathrm{otherwise}.
\end{array}\right.$$
\end{proposition}

In particular, since $H^*(G_2(\RR^n);\ZZ_2)$ is generated by the Stiefel-Whitney classes $w_1$ and $w_2$,
the above implies that 
$$w_1^{2^{s+1}-2}\cdot w_2^{n-2^s-1}\in H^{2(n-2)}(G_2(\RR^n);\ZZ_2)$$
is a non-trivial cup product of maximal length in $H^*(G_2(\RR^n);\ZZ_2)$ (cf.~note after Proposition
\ref{prop:g2n}).

\begin{proposition}(Stong, \cite[p.~104]{Stong}) 
\label{prop:products in g3n}
Let $2^s< n\le 2^{s+1}$. Then the following classes in $H^{3(n-3)}(G_3(\RR^n);\ZZ_2)$ ($1\le p$, $0<t<2^{p-1}$
are non-trivial products of maximal length in $H^*(G_3(\RR^n);\ZZ_2)$:
$$\begin{array}{ll}
\mathrm{if\ } n=2^{s+1}-2^p+1, & \mathrm{then\ \ } w_1^{2^{s+1}-2} \cdot w_2^{2^{s+1}-3\cdot 2^{p-1}-2}\ne 0; \\
\mathrm{if\ } n=2^{s+1}-2^p+1+t, & \mathrm{then\ \ } w_1^{2^{s+1}-1} \cdot w_2^{2^{s+1}-3\cdot 2^{p-1}-1}\cdot w_3^{t-1}\ne 0; \\
\mathrm{if\ } n=2^{s+1}, & \mathrm{then\ \ } w_1^{2^{s+1}-1} \cdot w_2^{2^{s+1}-4}\ne 0.
\end{array}$$
\end{proposition}

(Note that there is a small typo in case (2) of Proposition \cite[p.~104]{Stong}, where it is stated 
that the longest non-trivial product is $w_1^{2^{s+1}-1} \cdot w_2^{2^{s+1}-3\cdot 2^{p-1}}\cdot w_3^{t-1}$. 
The value for the cup-length in the same Proposition is correct, and the 
correct non-trivial product can be found in \cite[p.~112]{Stong}.)

\begin{proposition}(Stong, \cite[p.~104]{Stong})
\label{prop:products in g4n}
Let $2^s< n\le 2^{s+1}$. Then the following classes in $H^{4(n-4)}(G_4(\RR^n);\ZZ_2)$ 
(\,$0\le r<s$, $0\le t<2^r$) are non-trivial products of maximal length in $H^*(G_4(\RR^n);\ZZ_2)$:
$$\begin{array}{ll}
\mathrm{if\ } n=2^s+1, & \mathrm{then\ \ } w_1^{2^{s+1}-2} \cdot w_2^{2^s-5}\ne 0, \ 
w_1^{2^{s+1}-1} \cdot w_2^{2^s-7}\cdot w_3\ne 0;\\
\mathrm{if\ } n=2^s+2^r+1+t, & \mathrm{then\ \ } w_1^{2^{s+1}-2} \cdot w_2^{2^s+2^{r+1}-5}\cdot
 w_4^t\ne 0,\\

 & \mathrm{also\ } w_1^{2^{s+1}-1} \cdot w_2^{2^s+2^{r+1}-7}\cdot w_3\cdot w_4^t\ne 0 \mathrm{\ if\
  } r>0 ;
\end{array}$$
\end{proposition}

While the products listed in the above propositions give the best cohomological lower bounds for 
the Lusternik-Schnirelmann 
category of $G_k(\RR^n)$ for $k\le 4$ it turns out that there are other non-trivial products that lead to better estimates 
of topological complexity. We prove the relevant results in the rest of this section. Note that whenever 
we state that certain product of Stiefel-Whitney classes is non-trivial in the cohomology of $\gkn$, 
then the same product is also non-trivial in the cohomology of $G_k(\RR^m)$ for every $m\ge n$ (because the 
corresponding ideal of relations is smaller).

\begin{proposition}
\label{prop:g2n}
If $2^s<n\le 2^{s+1}$, then 
$$w_1^{2^s}\cdot w_2^{n-2^s-1}\ne 0 \text{\ \ in\ \ } H^{2(n-2)-2^s+2}(G_2(\RR^n);\ZZ_2)$$
and 
$$w_1^{2^s}\cdot w_2^{n-2^s}=0 \text{\ \ in\ \ } H^{2(n-2)-2^s+4}(G_2(\RR^n);\ZZ_2).$$
\end{proposition}  
\begin{proof}
By Proposition \ref{prop:prop pi*}(5), in order to check the first claim  we must show that 
$$\pi^*(w_1^{2^s}\cdot w_2^{n-2^s-1})\cdot e_1\cdot e_3^{n-3}\cdots e_{n-1}=
(e_1^{2^s}+e_2^{2^s})(e_1\cdot e_2)^{n-2^s-1}\cdot e_1\cdot e_3^{n-3}\cdots e_{n-1}=$$
$$=e_1^{n-2^s}\cdot e_2^{n-1}\cdot e_3^{n-3}\cdots e_{n-1}\ne 0
\text{\ \ in\ \ } H^{2(n-2)-2^s+2}(\fln;\ZZ_2).$$
By multiplying the result by $e_1^{2^s-2}$ we get $e_1^{n-2}\cdot e_2^{n-1}\cdot e_3^{n-3}\cdots e_{n-1}$, 
which is non-trivial by Proposition \ref{prop:prop pi*}(3).

Similarly, for the second claim we compute 
$$\pi^*(w_1^{2^s}\cdot w_2^{n-2^s})\cdot e_1\cdot e_3^{n-3}\cdots e_{n-1}=
(e_1^{2^s}+e_2^{2^s})(e_1\cdot e_2)^{n-2^s}\cdot e_1\cdot e_3^{n-3}\cdots e_{n-1}=0,$$
because $e_1^n=e_2^n=0$ by Proposition \ref{prop:prop pi*}(2).
\end{proof}

Note that a similar computation shows that $w_1^{2^{s+1}-2}\cdot w_2^{n-2^s-1}\ne 0$ in $
H^{2(n-2)}(G_2(\RR^n);\ZZ_2)$, and that the second claim implies that the exponent of $w_1$ in every other 
non-trivial product in $H^{2(n-2)}(G_2(\RR^n);\ZZ_2)$ must be smaller than $2^s$.

Next we turn to cohomology of $G_3(\RR^n)$ in the range $2^s<n\le 2^{s+1}$.

\begin{proposition}
\label{prop:g3n}
\ \\[-6mm]
\begin{enumerate}
\item 
$w_1^{2^s}\cdot w_2^{2^{s-1}}\ne 0 \text{\ \ and \ \ } w_1^{2^s}\cdot w_2^{2^{s-1}}\cdot w_3=0
\text{\ \ in \ \ }  H^*(G_3(\RR^{2^s+1});\ZZ_2);$
\\
\item 
$w_1^{2^s}\cdot w_2^{2^{s-1}}\cdot w_3\ne 0 \text{\ \ and \ \ } w_1^{2^s}\cdot w_2^{2^{s-1}}\cdot w_3^2=0
\text{\ \ in \ \ }  H^*(G_3(\RR^{2^s+2});\ZZ_2);$
\\
\item 
if $3\le t\le 2^s$, then $w_1^{2^s}\cdot w_2^{2^s}\cdot w_3^{t-3}\ne 0 \text{\ \ in \ \ }  H^{3(2^s+t-3)}(G_3(\RR^{2^s+t});\ZZ_2)$.
\end{enumerate}
\end{proposition}
\begin{proof}
Toward the proof of (1), by Proposition \ref{prop:prop pi*}(5), we must show that the product  
$$\pi^*(w_1^{2^s}\cdot w_2^{2^{s-1}})\cdot e_1^2\cdot e_2\cdot e_4^{2^s-3}\cdots e_{2^s}\ne 0.$$
To obtain a top-dimensional class we multiply the above product by $e_1\cdot e_2$ and take into account that
$e_1^{2^s+1}= e_2^{2^s+1}=0$ to obtain
$$e_1\cdot e_2\cdot\pi^*(w_1^{2^s}\cdot w_2^{2^{s-1}})\cdot e_1^2\cdot e_2\cdot e_4^{2^s-3}\cdots e_{2^s}=$$
$$=e_1\cdot e_2\cdot (e_1^{2^s}+e_2^{2^s}+e_3^{2^s})\cdot \big((e_1\cdot e_2)^{2^{s-1}}+(e_1\cdot e_3)^{2^{s-1}}+
(e_2\cdot e_3)^{2^{s-1}}\big)\cdot e_1^2\cdot e_2\cdot e_4^{2^s-3}\cdots e_{2^s}=$$
$$=e_1^{2^s-2}\cdot e_2^{2s-3}\cdot e_3^{2^s-1}\cdot e_4^{2^s-3}\cdots e_{2^s},$$
which is non-trivial by Proposition \ref{prop:prop pi*}(3), therefore $w_1^{2^s}\cdot w_2^{2^{s-1}}\ne 0$.

On the other hand, it is easy to check that all terms in the expansion of the product 
$$\pi^*(w_1^{2^s}\cdot w_2^{2^{s-1}}\cdot w_3)\cdot e_1^2\cdot e_2\cdot e_4^{2^s-3}\cdots e_{2^s}=$$
$$=(e_1^{2^s}+e_2^{2^s}+e_3^{2^s})\cdot \big((e_1\cdot e_2)^{2^{s-1}}+(e_1\cdot e_3)^{2^{s-1}}+
(e_2\cdot e_3)^{2^{s-1}}\big)\cdot (e_1\cdot e_2\cdot e_3)\cdot e_1^2\cdot e_2\cdot e_4^{2^s-3}\cdots e_{2^s}$$
have at least one factor whose exponent is bigger or equal to $2^s+1$, so the entire expression is trivial
by Proposition \ref{prop:prop pi*}(2).

Case (2) is entirely analogous and is left to the reader.

For case (3) we compute 
$$\pi^*(w_1^{2^s}\cdot w_2^{2^s}\cdot w_3^{t-3})\cdot e_1^2\cdot e_2\cdot e_4^{2^s-3}\cdots e_{2^s}=$$
$$(e_1^{2^s}+e_2^{2^s}+e_3^{2^s})\cdot \big((e_1\cdot e_2)^{2^s}+(e_1\cdot e_3)^{2^s}+
(e_2\cdot e_3)^{2^s}\big)\cdot (e_1\cdot e_2\cdot e_3)^{t-3}\cdot e_1^2\cdot e_2\cdot e_4^{2^s-3}\cdots e_{2^s}=$$
$$=e_1^{2^s+t-1}\cdot e_2^{2^s+t-2}\cdot e_3^{2^s+t-3}\cdot e_4^{2^s-3}\cdots e_{2^s},$$
which is non-trivial by Proposition \ref{prop:prop pi*}(5). Note that $w_1^{2^s}\cdot w_2^{2^s}\cdot w_3^{t-3}$
is a top-dimensional class in the cohomology of $G_3(\RR^{2^s+t})$.
\end{proof}

For $k>3$ the computations become increasingly complicated but we are still able to obtain some
general results. 

\begin{proposition}
\label{prop:w1w2w3}
\ \\[-6mm]
\begin{enumerate}
\item
If $3\le k\le 2^s$, then $w_1^{2^s}\cdot w_2^{2^s}=w_3^{2^s}\ne 0$ in $H^{3\cdot 2^s}(G_k(\RR^{2^s+k});\ZZ_2)$.\\
\item
If $3\le k\le 2^{s-1}$, then  $w_1^{2^s}\cdot w_2^{2^s}\cdot w_3^{2^{s-1}}\ne 0$ in 
$H^{3(2^s+2^{s-1})}(G_k(\RR^{2^s+2^{s-1}+k});\ZZ_2)$.\\
\item
If $3\le k\le 2^s$, then  $w_1^{2^s}\cdot w_2^{2^s}\cdot w_3^{2^s}=0$ in 
$H^{6\cdot 2^s}(G_k(\RR^{2^s+k});\ZZ_2)$.

\end{enumerate}
\end{proposition}
\begin{proof}
First note that 
$$\pi^*(w_1^{2^s}\cdot w_2^{2^s})=\left( \sum_{1\le a\le k} e_a^{2^s} \right) \cdot\left(\sum_{1\le b<c\le k} e_b^{2^s}e_c^{2^s}\right)=\left(\sum_{1\le a< b<c\le k} e_a^{2^s} e_b^{2^s}e_c^{2^s}
\right)=\pi^*(w_3^{2^s})$$
because $e_i^{2^{s+1}}=0$ by Proposition 
\ref{prop:prop pi*}(2), therefore all products where $a$ equals $b$ or $c$ disappear. Injectivity of $\pi^*$
implies that $w_1^{2^s}\cdot w_2^{2^s}=w_3^{2^s}$. 

To prove that the product is non-trivial we must show that
$$\pi^*(w_3^{2^s})\cdot (e_1^{k-1}\cdots e_{k-1})\cdot 
(e_{k+1}^{2^s-1}\cdots e_{2^s+k-1})$$
is a non-zero element of $H^*(\mathrm{Flag}(\RR^{2^s+k});\ZZ_2)$. To obtain a top-dimensional 
element we multiply the above product by $e_4^{2^s}\cdots e_k^{2^s}$:
$$\pi^*(w_3^{2^s})\cdot (e_1^{k-1}\cdots e_{k-1})\cdot 
(e_{k+1}^{2^s-1}\cdots e_{2^s+k-1})\cdot (e_4^{2^s}\cdots e_k^{2^s})=$$
$$=\left(\sum_{1\le a< b<c\le k} e_a^{2^s} e_b^{2^s}e_c^{2^s}
\right)\cdot (e_1^{k-1}\cdots e_{k-1})\cdot 
(e_{k+1}^{2^s-1}\cdots e_{2^s+k-1})\cdot (e_4^{2^s}\cdots e_k^{2^s})=$$
$$=e_1^{2^s+k-1}e_2^{2^s+k-2}\cdots e_k^{2^s} e_{k+1}^{2^s-1}\cdots e_{2^s+k-1}\ne 0.$$
Note that multiplication by $e_4^{2^s}\cdots e_k^{2^s}$ annihilates all summands in line two except 
$e_1^{2^s} e_2^{2^s} e_3^{2^s}$.

The computation in case (2) is analogous so we leave the details to the reader: 
$$\pi^*(w_1^{2^s}\cdot w_2^{2^s}\cdot w_3^{2^{s-1}})\cdot (e_1^{k-1}\cdots e_{k-1})\cdot 
(e_{k+1}^{2^s+2^{s-1}-1}\cdots e_{2^s+2^{s-1}+k-1})\cdot (e_4^{2^s+2^{s-1}}\cdots 
e_k^{2^s+2^{s-1}})=$$
$$=e_1^{2^s+2^{s-1}+k-1}e_2^{2^s+2^{s-1}+k-2}\cdots e_k^{2^s+2^{s-1}} 
e_{k+1}^{2^s+2^{s-1}-1}\cdots e_{2^s+2^{s-1}+k-1}\ne 0.$$

To prove (3) note that 
$$\pi^*(w_1^{2^s}\cdot w_2^{2^s}\cdot w_3^{2^s})=\pi^*(w_3^{2^{s+1}})=\left(\sum_{1\le a< b<c\le k} e_a^{2^{s+1}} 
e_b^{2^{s+1}}e_c^{2^{s+1}}\right)=0$$
because $e_i^{2^{s+1}}=0$.
\end{proof}

We conclude with a result that for $G_4(\RR^n)$ improves Proposition \ref{prop:w1w2w3} in a range of 
dimensions.

\begin{proposition}
\label{prop:g4n}
If $3\le t\le 2^{s-1}$, then $w_1^{2^s}\cdot w_2^{2^s}\cdot w_3^{2^{s-1}}\!\!\cdot w_4^{t-3}\ne 0$ in 
$H^*(G_4(\RR^{2^s+2^{s-1}+t});\ZZ_2)$.
\end{proposition}
\begin{proof}
To simplify the notation let us denote $d:=2^s+2^{s-1}$. Then we compute the value of the following 
top-dimensional element:
$$(e_1)^{d-4}\cdot \pi^*(w_1^{2^s}\cdot w_2^{2^s}\cdot w_3^{2^{s-1}}\cdot w_4^{t-3})\cdot
e_1^3\cdot e_2^2\cdot e_3 \cdot e_5^{d+t-5}\cdots e_{d+t-1}=$$
$$=\pi^*(w_3^d)\cdot (e_1\cdot e_2\cdot e_3\cdot e_4)^{t-3}\cdot 
e_1^{d+t-4}\cdot e_2^{t-1}\cdot e_3^{t-2}\cdot e_4^{t-3} \cdot e_5^{d+t-5}\cdots e_{d+t-1}=$$
$$=e_1^{d+t-4}\cdot e_2^{d+t-1}\cdot e_3^{d+t-2}\cdot e_4^{d+t-3} \cdot e_5^{d+t-5}\cdots e_{d+t-1}\ne 0$$
(the other summands vanish because the exponent of $e_1$ is at least $d+t$).
\end{proof}

\ \\

\section{Topological complexity of Grassmann manifolds}
\label{sec:TC of gkn}

Most of this section is devoted to the computation of lower bounds for the topological 
complexity of real Grassmannians but let us first mention the complex case which is much simpler: 
complex Grassmann
manifolds $G_k(\CC^n)$ are simply-connected and symplectic, so by Proposition \ref{prop:symplectic}
we have the precise value 

\begin{proposition}
Topological complexity of complex Grassmann manifold is $\TC(G_k(\CC^n))=2k(n-k)+1$.
\end{proposition}

The fundamental group of real Grassmannians is $\ZZ_2$, so we may apply Proposition 
\ref{prop:small pi1} to obtain upper bounds for topological complexity. 

\begin{theorem}
\label{thm:upper bound}
Topological complexity of real Grassmannians is bounded above by 
$$\TC(\gkn)\le 2 k(n-k). $$
In fact, unless $k=1, n=2^d$ or $k=2, n=2^d+1$ we have a better estimate 
$$\TC(\gkn)\le 2 k(n-k)-1.$$
\end{theorem}
\begin{proof}
By \cite{Berstein} the height of the class $w_1$ that generate $H^1(\gkn;\ZZ_2)$ is strictly smaller
than the dimension, except if $k=1$ or if $k=2$ and $n=2^d+1$. 
Moreover, if $k=1$, then by
\cite{FTY:TC of RPn} $\TC(\RR P^n)=2n$ if, and only if, $n$ is a power of 2, otherwise 
$\TC(\RR P^n)\le 2n-1$.
We don't know whether $\TC(G_2(\RR^{2^d+1}))$ can be equal to the stated upper bound $4(2^d-1)$. 
We will see below that cohomological lower estimates for the topological complexity are much smaller.
\end{proof}

By Proposition \ref{prop:TC properties}(5), in order to estimate from below topological complexity of 
$X$ we need a supply of elements in 
$$\Ker (\Delta^*\colon H^*(X\times X;\ZZ_2)\to H^*(X;\ZZ_2)).$$
For every $w\in H^*(X;\ZZ_2)$ we have 
$$\Delta^*(w\times 1+1\times w)=w\smile 1+1\smile w=w+w=0,$$
therefore 
$$z(w):=w\times 1+1\times w\in\Ker\Delta^*.$$
Farber \cite{Farber:TC} called elements of $\Ker\Delta^*$ \emph{zero-divisors} and defined 
\emph{zero-divisor cup-length} of $X$, denoted $\zcl(X)$, to be the maximal number of elements 
of $\Ker\Delta^*$ whose product is non-trivial (by analogy with the classical \emph{cup-length} from LS-
category theory, cf. \cite{CLOT}). In other words, $\zcl(X)=\nil(\Ker\Delta^*)-1$. Note that in general
one can define $\zcl(X)$ with respect to any multiplicative cohomology theory, but in this paper we consider
only ordinary cohomology with $\ZZ_2$-coefficients.

To determine the zero-divisor cup-length of Grassmannians we will need information about the 
height of basic zero-divisors of the form $z(w)$ and this can be easily expressed in terms of the 
height of $w$. For an integer 
$n\ge 0$ let $\rho(n)$ denote the minimal integral power of 2 that is strictly bigger than $n$ 
(e.g., $\rho(8)=\rho(11)=16$). In particular $\frac{\rho(n)}{2}\le n < \rho(n)$.

\begin{lemma}
\label{lem:height}
For every $w\in H^*(X;\ZZ_2)$ 
$$\mathrm{height}(z(w))=\rho(\mathrm{height}(w))-1.$$
\end{lemma}
\begin{proof}
Let $n=\height(w)$, i.e. $w^n\neq 0$ and $w^{n+1}=0$. It is well-known that all binomial 
coefficients of the form ${2^s-1\choose k}$, $k=0,1,\ldots,2^s-1$ are odd and that all 
binomial coefficients of the form ${2^s\choose k}$, $k=1,\ldots,2^s-1$ are even.
As a consequence
$$z(w)^{2^s-1}=(w\times 1+1\times w)^{2^s-1}$$
has a non-zero coefficient at the term $w^n\times w^{2^s-1-n}\ne 0$. Therefore if, $\rho(n)=2^s$, then
$z(w)^{2^s-1}\ne 0$.
On the other hand 
$$z(w)^{2^s}=(w\times 1+1\times w)^{2^s}=(w^{2^s}\times 1+1\times w^{2^s})=0.$$
\end{proof}

\begin{lemma}
\label{lem:zcl}
If for some $u_1,\ldots,u_n\in H^*(X;\ZZ_2)$ and positive integers $k_1,\ldots,k_n$
$$u_1^{k_1}\cdot\ldots\cdot u_n^{k_n}\ne 0,$$
then 
$$z(u_1)^{\rho(k_1)-1}\cdot\ldots\cdot z(u_n)^{\rho(k_n)-1}\ne 0.$$
\end{lemma}
\begin{proof}
It is sufficient to observe that the term 
$$\big( u_1^{k_1}\cdot\ldots\cdot u_n^{k_n}\big) \times 
\big( u_1^{\rho(k_1)-k_1-1}\cdot\ldots\cdot u_n^{\rho(k_n)-k_n-1}\big)$$
appears with a non-zero coefficient in the product 
$z(u_1)^{\rho(k_1)-1}\cdot\ldots\cdot z(u_n)^{\rho(k_n)-1}$, and that it
is non-trivial, because $\rho(k_i)-k_i-1\le k_i$ for $i=1,\ldots,n$.
\end{proof}

The question of the topological complexity of $G_1(\RR^n)=\RR P^{n-1}$ is considered to be solved 
by being reduced to the classical 
immersion problem for projective spaces (see \cite{FTY:TC of RPn}), so we pass to 
the next case and derive a lower bound for the topological complexity of $G_2(\RR^n)$ (recall our standing 
assumption that $k\le n/2$).

\begin{theorem}
\label{thm:g2n}
Assume $2^s<n\le 2^{s+1}$. Then 
$$z(w_1)^{2^{s+1}-1}\cdot z(w_2)^{\rho(n-1-2^s)-1}\ne 0$$
is a non-trivial product of zero-divisors of maximal length, so that 
$$\zcl((G_2(\RR^n))=2^{s+1}+\rho(n-1-2^s)-2$$ 
and
$$\TC(G_2(\RR^n))\ge 2^{s+1}+\rho(n-1-2^s)-1.$$
\end{theorem}
\begin{proof}
We proved in Proposition \ref{prop:g2n} that 
$w_1^{2^s}\cdot w_2^{n-2^s-1}\ne 0$ and that $w_1^{2^s}\cdot w_2^{n-2^s}=0$ in the cohomology of $G_2(\RR^n)$.
Non-triviality of the first product implies that zero-divisor cup-length is at least 
$$\zcl(G_2(\RR^n))\ge \rho(2^s)-1+\rho(n-2^s-1)-1=2^{s+1}+\rho(n-2^s-1)-2.$$
On the other hand, triviality of the second product implies that if a non-trivial product 
$w_1^a\cdot w_2^b\ne 0$ has $b\ge n-2^s$, then $a<2^s$, which yields a shorter product 
of zero-divisors (indeed, $\height(w_2)=n-2$ by \cite{DuttaKhare2002} and so $\rho(w_2)\le 2^{s+1}$). 
Therefore $z(w_1)^{2^{s+1}-1}\cdot z(w_2)^{\rho(n-1-2^s)-1}$ is the longest non-trivial product of zero
divisors.

Estimate for topological complexity follows from Proposition \ref{prop:TC properties}(4), as
$\TC(X)\ge\zcl(X)+1$.
\end{proof}

Observe that, by the above argument, in order to achieve maximal zero-divisor cup-length one must look 
for products of the form 
$w_1^a\cdot w_2^b\cdot w_3^c \cdots$ where $a,b,c,\ldots$ are powers of two (and as big as possible). 
This is why products with maximal cup-length as in \cite{Stong} often don't give products with maximal
zero-divisor cup-length.

\begin{example}
The first non-trivial case to consider is $G_2(\RR^4)$. It is 4-dimensional and since $w_1^4=0$ 
(Proposition \ref{prop:height w1}), Proposition \ref{prop:small pi1} implies that  
$\TC(G_2(\RR^4))\le 7$.
A lower estimate can be computed from the above theorem as 
$\TC(G_2(\RR^4))\ge \rho(3)+\rho(1)-1=5$, therefore $5\le \TC(G_2(\RR^4))\le 7$.
\end{example}

\begin{example}
Note that the gap between the upper bound of Theorem \ref{thm:upper bound} and 
the lower bound of Theorem \ref{thm:g2n} grows with $n$. For example, by applying mentioned results
we obtain 
$$23 \le\TC(G_2(\RR^{13}))\le 43.$$
This is perhaps not surprising in view of the irregularities and gaps between the lower estimates 
and the dimensional upper bounds for the topological complexity of projective spaces. 
As we mentioned before, computation of $\TC(\RR P^n)$ was reduced to the famous immersion problem
for projective spaces - see tables in  https://www.lehigh.edu/~dmd1/imms.html, whose 
computation has required an immense body of work. 
Nevertheless, estimates of Theorem \ref{thm:g2n} represent a notable improvement with respect
to those implied by Propositions \ref{prop:TC properties}(3) and \ref{prop:LScat properties}(4)
as they would give only $\TC(G_2(\RR^{13}))\ge 19$. 
\end{example}

To estimate topological complexity of $G_3(\RR^n)$ we use Proposition \ref{prop:g3n} and consider
three cases. We remind the reader that if a 
product of Stiefel-Whitney classes is non-zero in $H^*(\gkn;\ZZ_2)$ then the same product is 
non-zero in $H^*(G_k(\RR^m);\ZZ_2)$ for all $m\ge n$.

\begin{theorem}
\label{thm:g3n}
Assume $2^s<n\le 2^{s+1}$ and let $t:=n-2^s$. 
The results on zero-divisor cup-length and topological complexity of $G_3(\RR^n)$ are summarized 
in the following table:
$$
\begin{array}{c|c|c|c}
t & \zcl(G_3(\RR^n)) & \text{realized by} & \TC(G_3(\RR^n))\ge\\[1mm]
\hline\\[-3mm]
1                & 3\cdot 2^s-2       & z(w_1)^{2^{s+1}-1}\!\cdot z(w_2)^{2^{s}-1}    & 3\cdot 2^s-2 \\[1mm]
2                & 3\cdot 2^s-1       & z(w_1)^{2^{s+1}-1}\!\cdot z(w_2)^{2^{s}-1}\! \cdot z(w_3)& 3\cdot 2^s-1 
\\[1mm]
3,\ldots,2^s & 4\cdot 2^s+\rho(t-3)-3 & z(w_1)^{2^{s+1}-1}\!\cdot z(w_2)^{2^{s+1}-1}\!\cdot 
z(w_3)^{\rho(t-3)-1} &
4\cdot 2^s+\rho(t-3)-2\\[1mm]
\end{array}
$$
Note that for about half of the range ($n\ge 2^s+2^{s-1}+3$) the exact value is $\zcl(G_3(\RR^n))=5\cdot 2^s-3$.
\end{theorem}
\begin{proof}
The three cases correspond to the points (1),(2) and (3) of Proposition \ref{prop:g3n}. In each of those we 
proved that certain products in the cohomology of $G_3(\RR^n)$ are non-trivial, and that products that 
would give longer products of zero-divisors are trivial. The values in the table follow by application of 
Lemma \ref{lem:zcl} and Proposition \ref{prop:TC properties}(5).

In particular, if $n\ge 2^s+2^{s-1}+3$, then $\rho(n-2^s-3)=2^s$, therefore $\zcl(G_3(\RR^n))=5\cdot 2^s-3$.
\end{proof}

\begin{example}
\label{ex:g3}
For $G_3(\RR^{11})$ the above theorem yields the estimate $\TC(G_3(\RR^{11}))\ge 31$, which is again 
a considerable improvement from the category lower bound $\TC(G_3(\RR^{11}))\ge 20$.
\end{example}

To describe the estimates of $\TC(\gkn)$ for $k>3$ we will consider for simplicity only 'generic' pairs $(k,n)$.

\begin{theorem}
\label{thm:gkn}
Let $4\le k\le 2^{s-1}$ and $2^s+k\le n\le 2^{s+1}$. Then 
$$\TC(\gkn)\ge 4\cdot 2^s-1 \ \   \text{for}\ \ n\le 2^s+2^{s-1}+2$$ 
and 
$$\TC(\gkn)\ge 5\cdot 2^s-2 \ \ \text{for}\ \ n\ge 2^s+2^{s-1}+3.$$
\end{theorem}
\begin{proof}
In the first case $w_1^{2^s}\cdot w_2^{2^s}\ne 0$ by Proposition \ref{prop:w1w2w3}(1), while 
in the second case $w_1^{2^s}\cdot w_2^{2^s}\cdot w_3^{2^{s-1}}\ne 0$ by Proposition \ref{prop:w1w2w3}(2).
Then the stated estimates follow by Proposition \ref{prop:TC properties}(5) and Lemma \ref{lem:zcl}.
Note that we do not claim that the mentioned products give products of zero-divisors of maximal length.
\end{proof}

\begin{example}
By the above theorem we have $\TC(G_5(\RR^{13}))\ge 31$, compared to the category estimate 
$\TC(G_5(\RR^{13}))\ge 16$, and $\TC(G_6(\RR^{27}))\ge 78$, compared to the category estimate 
$\TC(G_6(\RR^{27}))\ge 32$.
\end{example}

The final result of this section is a direct consequence of Proposition \ref{prop:g4n} and gives 
the zero-divisor cup-length of $G_4(\RR^n)$ for  approximately half of dimensions. The proof follows 
the lines of the previous theorems and is thus left to the reader. 

\begin{theorem}
\label{thm:g4n}
Assume $2^s+2^{s-1}+3\le n \le 2^{s+1}$ and let $t:=n-2^s-2^{s-1}$. Then
$$\zcl(G_4(\RR^n))=5\cdot 2^s+\rho(t-3)-4,$$
realized by the product
$$z(w_1)^{2^{s+1}-1}\!\cdot z(w_2)^{2^{s+1}-1}\!\cdot z(w_3)^{2^{s}-1}\!\cdot 
z(w_4)^{\rho(t-3)-1}\ne 0.$$
\end{theorem}

\ \\

\section{Monotonicity of category and topological complexity for increasing sequences of Grassmannians}
\label{sec:monotonicity}

We already mentioned that $\cat(G_1(\RR^n))=\cat(\RR P^{n-1})=n$, so clearly $m\le n$ implies that
$\cat(G_1(\RR^m))\le \cat(G_1(\RR^n))$. This is not a coincidence, in fact $\RR P^m$ is the
$m$-skeleton of $\RR P^n$ with respect to the standard CW-decomposition of projective spaces
and it is a well-known fact (see \cite[Theorem 1.66]{CLOT}) that $\cat(X^{(m)})\le\cat(X)$ for
every non-contractible CW-complex $X$. The analogous comparison of topological complexities 
of projective spaces is less obvious. Farber, Tabachnikov and Yuzvinsky \cite{FTY:TC of RPn}
proved that for $n\ne 1,3,7$ $\TC(\RR P^n)$ equals the minimal $k$ such that $\RR P^n$ admits
an immersion in $\RR^{k-1}$. Therefore, $\TC(G_1(\RR^n))$ is clearly an increasing function of $n$
(see Example \ref{ex:TC RPn} for a more direct proof). 

Intuitively, the spaces $G_k(\RR^n)$ become more complicated by increasing the dimension, so one
should naturally expect that their category and topological complexity increase as well. 
But this is not always the case, as we can see from the sequence of spheres
$S^1\subseteq S^2\subseteq S^3\subseteq \ldots \subseteq  S^{\infty}$
whose topological complexities form the alternating sequence $2,3,2,3,\ldots$ and $\TC(S^\infty)=1$.
The case of Grassmannians is even more complicated because (for $k>1$) $G_k(\RR^m)$ is not even
close to being a skeleton of $G_k(\RR^n)$. In spite of that we will show that for $k=2,3$ 
the category of $G_k(\RR^n)$ is indeed increasing with $n$. The case of topological complexity 
is more complicated but we will still manage to prove some partial results. 

Throughout this section we will view $G_k(\RR^m)$ as a subcomplex of $G_k(\RR^n)$ 
($2k\le m\le n$) with respect to the standard decomposition described in Section 
\ref{sec:Grassmannians}. Our first observation is that the cells of the respective 
decompositions coincide up to dimension $m-k$, that is to say 
$$G_k(\RR^n)^{(m-k)}= G_k(\RR^m)^{(m-k)}.$$
In \cite{Pavesic:Monotonicity} we proved the following result:

\begin{theorem}(\cite[Theorem 3.6]{Pavesic:Monotonicity})
Let $X$ be a connected CW-complex $X$, and let $A$ be a subcomplex of $X$ containing
$X^{d}$. If $\dim(A)< d+\cat(X)-1$, then $\cat(A)\le\cat(X)$.
\end{theorem}

In our case we obtain that $\cat(G_k(\RR^m))\le \cat(G_k(\RR^n))$
provided that $k(m-k)< m-k + \cat(G_k(\RR^n))-1$, or equivalently, if 
$\cat(G_k(\RR^n))> (k-1)(m-k)+1$.

\begin{theorem}
If $k\le 3$ and $2k\le m\le n$, then
$\cat(G_k(\RR^m))\le \cat(G_k(\RR^n))$.
\end{theorem}
\begin{proof}
The case $k=1$ is clear, so let us consider $k=2$. By the above discussion it is sufficient
to show that $\cat(G_2(\RR^n))> n-2$ for all $n\ge 4$. By Proposition \ref{prop:height w1}, if
$2^s< n\le 2^{s+1}$, then $w_1^{2^{s+1}-2}\cdot w_2^{n-2^s-1}\ne 0$, therefore
$$\cat(G_2(\RR^n)) \ge 2^{s+1}-2+n-2^s=n+2^s-2> n-2.$$

If $k=3$, we must check that $\cat(G_3(\RR^n))> 2n-5$ for all $n\ge 6$. To do this, we consider
the three cases of Proposition \ref{prop:products in g3n}. If $n=2^{s+1}-2^p+1+t$ ($1\le p<s,\
0<t<2^{p-1}$), then 
$$\cat(G_3(\RR^n))-(2n-5)\ge (2^{s+1}-1+2^{s+1}-3\cdot 2^{p-1}-1+t)-2(2^{s+1}-2^p+1+t)+5=$$
$$=2^{p-1}+1-t>0.$$
The remaining two cases are verified similarly.
\end{proof}

For $k=4$ our approach works only under some additional assumptions.

\begin{theorem}
If $m\le 2^s+1$, then $\cat(G_4(\RR^m))\le \cat(G_4(\RR^{2^s+1}))$.

Otherwise, if $n=2^s+2^p+1+t$ ($0\le p< s,\ 0\le t<2^p$), then $\cat(G_4(\RR^m))\le 
\cat(G_4(\RR^n))$, provided that $n-m > \frac{2^p+2t-2}{3}$.
\end{theorem}
\begin{proof}
To prove the claim we must show that $\cat(G_4(\RR^m))> 3m-11$. In the first case we have
by Proposition \ref{prop:products in g4n} that $w_1^{2^{s+1}-2} \cdot w_2^{2^s-5}\ne 0$, therefore  
$$\cat(G_4(\RR^{2^s+1}))\ge 2^{s+1}+2^s-6=3\cdot (2^s+1)-9>3m-11$$
because $m\le 2^s+1$.

Similarly, if $n=2^s+2^p+1+t$, then $w_1^{2^{s+1}-2} \cdot w_2^{2^s+2^{r+1}-5}\cdot w_4^t\ne 0$,
therefore 
$$\cat(G_4(\RR^r))\ge 2^{s+1}+2^s+2^{p+1}+t-6=3\big(2^s+2^p+1+t-\frac{2^p+2t-2}{3}\big)-11>
3m-11.$$
\end{proof}

\begin{example}
By comparing estimates for $\cat(G_4(\RR^n))$ we find that 
$$\cat(G_4(\RR^8))\le \cat(G_4(\RR^9))\le \ldots \le\cat(G_4(\RR^{13}))$$
but we cannot conclude that $\cat(G_4(\RR^{13}))\le \cat(G_4(\RR^{14}))$, In fact, we can only 
claim a weaker estimate $\cat(G_4(\RR^{12}))\le \cat(G_4(\RR^{14}))$.
\end{example}

In order to derive analogous results for topological complexity we rely on the following 
result from \cite{Pavesic:Monotonicity}.

\begin{theorem}(\cite[Theorem 3.1]{Pavesic:Monotonicity}). 
Let $X$ be a connected CW-complex $X$, and let $A$ be a subcomplex of $X$ containing
$X^{d}$. If $2  \dim(A)< d+\TC(X)-1$, then $\TC(A)\le\TC(X)$.
\end{theorem}

In the case of the inclusion $G_k(\RR^m)\subseteq G_k(\RR^n)$ the condition becomes 
$$2k(m-k)<(m-k)+\TC(G_k(\RR^n))-1,$$
so we have

\begin{corollary}
Let $2k\le m\le n$. If $\TC(G_k(\RR^n))> (2k-1)(m-k)+1$, then $\TC(G_k(\RR^m))\le \TC(G_k(\RR^n))$.
\end{corollary}

\begin{example}
\label{ex:TC RPn}
Let $k=1$ and $m<n$. Then we have 
$$\TC(G_1(\RR^m))\ge \cat(G_1(\RR^n))=n > (2\cdot 1-1)(m-1)+1,$$ 
therefore
$\TC(G_1(\RR^m))\le \TC(G_1(\RR^n))$ whenever $m\le n$, as we already mentioned at the beginning 
of this section.
\end{example}

\begin{example}
For $k=2$ the condition in the above Corollary becomes $\TC(G_2(\RR^n))>3m-5$. For instance,
we proved in Theorem \ref{thm:g2n} that $\TC(G_2(\RR^{2^s+1}))\ge 2^{s+1}$, therefore 
$\TC(G_2(\RR^{2^s+1}))\ge \TC(G_2(\RR^m))$ whenever $m<\frac{2^{s+1}+5}{3}$. Hence, we have 
$\TC(G_2(\RR^9))\ge \TC(G_2(\RR^6))$, but we do not know whether $\TC(G_2(\RR^9))\ge \TC(G_2(\RR^7))$.
\end{example}


\end{document}